\documentclass{amsart}

\usepackage{amsmath}

\usepackage{amssymb}

\usepackage{graphicx}

\newtheorem{thm}{Theorem}

\newtheorem{lemma}{Lemma}

\theoremstyle{definition}

\theoremstyle{remark}
\newtheorem{remark}{Remark}

\numberwithin{equation}{section}

\def\R{\mathbb{R}}

\def\E{\mathbb{E}}

\def\SO{\mathrm{SO}}

\def\P{\mathrm{P}}

\def\R{\mathbb{R}}
\def\P{\mathbb{E}(2,1)}
\def\p{\mathfrak{p}}

\def\FR{\mathbf{F}}
\def\k{\kappa}
\def\Re{\mathbb{R}^{2,1}}

\begin{document}

\title[Hamiltonian flows on null curves]{Hamiltonian flows on null curves}

\author{Emilio Musso}
\address{(E. Musso) Dipartimento di Matematica, Politecnico di Torino,
Corso Duca degli Abruzzi 24, I-10129 Torino, Italy}
\email{emilio.musso@polito.it}

\author{Lorenzo Nicolodi}
\address{(L. Nicolodi) Di\-par\-ti\-men\-to di Ma\-te\-ma\-ti\-ca,
Uni\-ver\-si\-t\`a degli Studi di Parma, Viale G. P. Usberti 53/A,
I-43100 Parma, Italy}
\email{lorenzo.nicolodi@unipr.it}

\thanks{Authors partially supported by MIUR projects:
\textit{Metriche riemanniane e variet\`a differenziabili} (E.M.);
\textit{Propriet\`a geometriche delle variet\`a reali e complesse}
(L.N.); and by the GNSAGA of INDAM}


%\date{Version of \today}

%\date{Version of June 28, 2010}


\begin{abstract}
The local motion of a null curve in Minkowski 3-space induces an
evolution equation for its Lorentz invariant curvature. Special
motions are constructed whose induced evolution equations are the
members of the KdV hierarchy. The null curves which move under the
KdV flow without changing shape are proven to be the trajectories of
a certain particle model on null curves described by a Lagrangian
linear in the curvature. In addition, we show that the curvature of
a null curve which evolves by similarities can be computed in terms
of the solutions of the second Painlev\'e equation.
\end{abstract}

\maketitle

\section{Introduction}\label{s:intro}

Many completely integrable nonlinear PDE (soliton equations)
describe the evolution of differential invariants associated with
curves moving in a homogeneous space (\cite{CIB-PhD}, \cite{CQ1},
\cite{CQ2}, \cite{DS}, \cite{GP}, \cite{H}, \cite{Ivey-cpm},
\cite{La2}, \cite{LP}, \cite{NSW}, \cite{Pi}). Some of these motions share the
property that the curves which evolve by congruences of the ambient
space have both a variational and a Hamiltonian description: as
extremals of a geometric variational problem defined by the
conserved integrals of the corresponding soliton equation, and as
solutions of an integrable (finite dimensional) contact Hamiltonian
system. Examples include
the modified Korteweg-de Vries (mKdV) equation
describing the (geodesic) curvature
evolution induced by a local motion of curves in
2-dimensional space-forms (\cite{GP}, \cite{M-re-torino}).
See also \cite{HS-PAMS02}, \cite{LS-siam}, \cite{LS3} for other examples.

\vskip0.1cm

In this paper, we investigate the local motion of null
curves in Minkowski 3-space and find local motions inducing the Korteweg-de
Vries (KdV) hierarchy of equations.
(The motion of spacelike and timelike
curves in Minkowski 3-space has recently been related to integrable equations
from the AKNS hierarchy in \cite{DI}.)
Our approach is similar in spirit to that used in \cite{GP},
where integrable equations
from the mKdV hierarchy are related to local motions of curves in the plane.
We then provide a variational description of null curves which move
by Lorentzian rigid motions under the KdV flow and show that they solve a finite
dimensional integrable Hamiltonian system. In this regard, motivations are provided
by recent studies on relativistic particle models on
null curves (\cite{FGL}, \cite{GM}, \cite{KP}, \cite{NMMK}).

\vskip0.1cm

Let $\Re$ be Minkowski 3-space and $\gamma \subset \Re$ a null curve
parametrized by the natural (pseudo-arc) parameter $s$ which
normalizes the derivative of its tangent vector field. It is known
that in general $\gamma$ is uniquely determined up to Lorentz
transformations by a Lorentzian invariant function
$\kappa_\gamma(s)$, called the curvature of $\gamma$ (\cite{BO},
\cite{Cas}, \cite{FGL}, \cite{In-Lee}).
We show (cf. Theorem \ref{thm:local-vf})
that the local motion of a null curve in Minkowski 3-space induces a
\textit{local} evolution equation for its curvature of the form
\begin{equation}\label{k-ev-eq}
   \frac{\partial\kappa}{\partial t} = \mathcal D D^{-1}\mathcal D \p[\kappa],
    \end{equation}
where $D$ is the total derivative operator with respect to $s$,
$\mathcal D = D^3 + 4\kappa D + 2\kappa_{s}$, and $\p[\kappa] =
\p(\kappa, \kappa_s, \dots)$ is a differential polynomial such that
$\kappa_s \p$ is a total derivative.\footnote{The study can actually be extended to
null curves in 3-dimensional Lorentzian space forms.}
Interestingly, the right hand
side of \eqref{k-ev-eq} is expressed in terms of the operators which
define the bi-Hamiltonian structure of the KdV equation (\cite{O2}).
For a particular sequence of differential functions $\p$, we show
that \eqref{k-ev-eq} coincide with the equations of the KdV
hierarchy, hence providing a new geometric interpretation of the KdV
flows (cf. Theorem \ref{thm:kdv hier}).

\vskip0.1cm

We then discuss the motion of null curves corresponding
to the traveling wave solutions of the KdV equation.
Such curves evolve under the KdV flow
by Lorentzian rigid motions, retaining their shape.
Theorem \ref{thm:cong} shows that they coincide with the critical points of the variational
problem on
null curves defined by the Lorentz invariant functional
\begin{equation}\label{02}
  \mathcal L_\lambda(\gamma) =\int_\gamma{(2\kappa_\gamma + \lambda)}ds,
   \quad \lambda\in \R.
   \end{equation}
Functionals of this type have been considered in the
literature as action functionals of natural geometrical particle models for null
trajectories in 3 and 4-dimensional spacetimes of constant
curvature (cf. \cite{CGR}, \cite{FGL}, \cite{KP},
\cite{NFS2}, \cite{NMMK}, \cite{P}, and the
references there). The integration of the worldlines
can be achieved by quadratures and the explicit formulae of the
natural parametrizations can be given in terms of
Weierstrass elliptic functions (\cite{GM}).
The main point in
the integration of the extremal curves is the existence of a Lax pair
encoding the Euler--Lagrange equations of
(\ref{02}). Remarkably, the Lax pair can be directly deduced
from the invariance of the trajectory with respect to the KdV dynamics.
Finally, in Theorem \ref{thm:sim}, we show that if the shape of a null curve
evolves by similarities under
the KdV flow then the curvature function can be integrated in terms of the solutions
of the second Painlev\'e equation. This provides a geometric
interpretation of the similarity reduction of the KdV equation
investigated in \cite{JO}.

\section{Preliminaries}\label{s:pre}

In this section we summarize some background material, referring to
\cite{O2}, \cite{O3} for a complete exposition.

Let $J_h (\R,\R) = \R \times \R^{h+1}$ denote the space of $h^{th}$ order
jets of smooth $\R$-valued functions $u$ of one independent variable $s$
with coordinates
$s$, $u_{(0)}$, $u_{(1)}$, $\dots$, $u_{(h)}$. The jet space $J_h (\R,\R)$
is endowed with the contact system
generated by the 1-forms
\[
 \zeta_j = du_{(j-1)} -u_{(j)}ds, \quad j= 1, \dots, h,
  \]
and independent condition $ds$.
The projective limit of the sequence
\begin{equation}\label{2.1.1}
 \cdots \to J_{h}(\R,\R)\to J_{h-1}(\R,\R)\to \cdots \to
  J_1(\R,\R)\to J_0(\R,\R)
   \end{equation}
is the {\it total jet space} of $\R$-valued smooth functions of one
independent variable. It is denoted by ${J}(\R,\R)$.
If $u:I\subset \R\to \R$ is a smooth function, its {\it prolongation of order}
$h$, $j_h(u): I \to J_h(\R,\R)$,  is the integral curve of the contact system
given by
$$
 j_h(u):   s \mapsto
  \left(s,u|_s,\frac{d u}{ds}\vert_s, \dots,\frac{d^h u}{ds^h}|_s \right).
  $$

A smooth map $\mathfrak w : J(\R,\R) \to \R$ is said a \textit{polynomial differential
function} (differential polynomial) of order $h$
if there exists a polynomial $w\in \R[x_0, \dots,x_h]$ such that
$$
   \mathfrak{w}(\mathbf{u})=w(u_{(0)},u_{(1)},\dots,u_{(h)}),
     $$
for each $\mathbf{u} = (s, u_{(0)}, u_{(1)}, \dots,u_{(h)}, \dots) \in J(\R,\R)$.
For a polynomial differential function $\mathfrak w$,
we will write $\mathfrak w[u]$ to remind that $\mathfrak w$ depends on $u$ and
the derivatives of $u$.
The algebra of polynomial differential functions, $J[\mathbf{u}]$,
is endowed with a derivation, called the {\it total derivative},
defined by
$$
 D\mathfrak w =
  \sum_{p=0}^{\infty} \frac{\partial w}{\partial u_{(p)}} u_{(p+1)}.
   $$

A differential function $\mathfrak w\in J[\mathbf{u}]$ is
a total derivative
if there exists $\mathfrak p\in J[\mathbf{u}]$ such
that $\mathfrak w=D\left(\mathfrak p\right)$.
The ``primitive'' $\mathfrak p$ is unique up to an additive constant.
By $\int \mathfrak w ds$ we denote the unique primitive of $\mathfrak w$
which vanishes at $\mathbf{u}=\mathbf{0}$.

On $J[\mathbf{u}]$ there is a natural differential operator
$E$, the Euler operator, defined by
\[
 E(\mathfrak w) =\sum _{\ell =0}^\infty (-1)^\ell D^\ell
 \left(\frac{\partial \mathfrak w}{\partial u_{(\ell)}}\right).
  \]
Note that $E(\mathfrak w)$ is the gradient of the functional
$\mathcal{W}$ defined by
\[
 \mathcal W : u \mapsto \int \mathfrak w[u]ds
  \]
and that it coincides with the variational derivative
$\frac{\delta \mathcal W}{\delta u}$.
We also recall a few more basic facts:
\begin{itemize}

\item $\mathfrak w \in J[\mathbf{u}]$ is a total derivative if and only if
$E(\mathfrak w) = 0$;

\item $u$ is an extremal of the functional $\mathcal W$ if and only if
$E(\mathfrak w)[u]=0$;

\item for each $\mathfrak w \in J[\mathbf{u}]$, $u_{(1)} E(\mathfrak w)$ is
a total derivative (cf. \cite{O3}, Theorem 7.36);

\item $\mathfrak w(u) = \int_0^1
E(\mathfrak w)[{\epsilon u}]\cdot u \,d\epsilon$. %vedi Lax
\end{itemize}

Next, consider the operator $\mathcal{D} : J[\mathbf{u}] \to J[\mathbf{u}]$,
defined by
\begin{equation}\label{2.2.1}
 \mathcal{D} = D^3 + 4u_{(0)}D + 2u_{(1)}.
    \end{equation}
According to \cite{GGKM} (cf. also \cite{Lx} and \cite{O2}), there exist two sequences
$\{\mathfrak g_n\}$ and $\{\mathfrak p_n\}$ of polynomial differential functions
satisfying the \textit{Lenard recursion formula}:
\begin{equation}\label{lenard-formula}
  D \mathfrak g_n = \mathcal{D} \mathfrak g_{n-1} ,
\quad  \mathfrak g_{n-1} = E (\mathfrak p_{n-1}), \quad n = 1, 2, \dots,
 \end{equation}
where
\[
  \mathfrak g_0=\frac{1}{2}, \quad \mathfrak p_0=\frac{1}{2}u_{(0)}.
   \]
The first $\mathfrak g_n$'s are:
\begin{eqnarray*}
\mathfrak g_1 = u, \quad \mathfrak g_2 =3 u^2+u_{ss} ,\quad
\mathfrak g_3 = 10 u^3+10 u_{ss} u+5 u_s^2+u_{ssss}.
\end{eqnarray*}

The Lenard recursion formula leads to the two Hamiltonian
representations of the KdV hierarchy of integrable evolution equations, namely
\begin{equation}\label{KdV-hierarchy}
  u_t = - D \mathfrak g_n [u] = - \mathcal{D} \mathfrak g_{n-1}[u],
   \quad n = 1, 2, \dots,
   \end{equation}
where the right hand sides depend only on $u$ and its derivatives
with respect to $s$.
The first member of the hierarchy is the wave equation
$u_{t} + u_{s}=0$, while the second is the Korteweg--de Vries equation
in the form
$$u_{t} = - 6 u u_{s} - u_{sss}.$$
The next member is the fifth order equation
$$
  u_{t} = -30 u_{s} u^2 - 10 u_{sss} u - 20 u_{s} u_{ss}-u_{5s}.
   $$

\begin{remark}
The $\mathfrak g_n$'s are the gradients of the Gardner--Kruskal--Miura \cite{GGKM}
sequence of conserved functionals
$$
  \mathcal P_n (u) = \int  \mathfrak p_n [u] ds
   $$
of the KdV hierarchy. The first three of them are:
\begin{eqnarray*}
\mathcal P_0 = \int \frac{u}{2} ds, \quad \mathcal P_1 = \int \frac{u^2}{2} ds ,\quad
\mathcal P_2 = \int \left(u^3 - \frac{u_s^2}{2} \right) ds.
\end{eqnarray*}
\end{remark}

\section{Null curves in Minkowski 3-space}

\subsection{Null curves and frames}

Let $\Re$ denote affine Minkowski 3-space with the Lorentzian inner
product
$$
  \langle \mathbf{x},\mathbf{y}\rangle = -(x^1y^3+x^3y^1) + x^2y^2 = x^ig_{ij}y^j.
    $$
We fix an orientation on $\Re$ by requiring that the standard basis
$(\mathbf{e}_1,\mathbf{e}_2,\mathbf{e}_3)$ is positive,
and fix a time-orientation by saying that a timelike or null (lightlike) vector
$\mathbf{x}\neq 0$ is future-directed if $\langle
\mathbf{x},\mathbf{e}_1+\mathbf{e}_3 \rangle < 0$. Let $\P = \Re \rtimes \SO_0(2,1)$
denote the restricted Poincar\'e group, i.e. the group of isometries of
$\Re$ preserving the given orientations. The elements of $\P$ can be
viewed as affine frames $(\mathbf x,\mathbf{a})$ consisting of a point
$\mathbf x\in \Re$ and a positive basis
$\mathbf{a}=(\mathbf{a}_1,\mathbf{a}_2,\mathbf{a}_3)$ such that
$\mathbf{a}_1,\mathbf{a}_3$ are future-directed null vectors and
\begin{equation}\label{o.n.conditions}
 \langle \mathbf{a}_i,\mathbf{a}_j\rangle = g_{ij},
  \quad i,j\in \{1,2,3\}.
    \end{equation}
We will think of $\P$ as the closed subgroup of $\mathrm{GL}(4,\R)$
whose elements are of the form
$$
 \mathbf{X}(\mathbf x,\mathbf{a})
  =\left(
     \begin{array}{cc}
            1 & 0 \\
         \mathbf x & \mathbf{a} \\
          \end{array}
           \right),
       $$
where $(\mathbf x,\mathbf{a})\in \P$. Correspondingly, the Lie algebra
$\mathfrak e(2,1)$ of $\P$ is the subalgebra of $\mathfrak{gl}(4,\R)$ of all
$4\times 4$ matrices of the form
$$
 X(q,v)=
  \begin{pmatrix}
   0 & 0 & 0 & 0\\
    q^1 & v^2 & v^3 & 0\\
     q^2 & v^1 & 0 & v^3\\
      q^3 & 0 & v^1 & -v^2
       \end{pmatrix}.
        $$

A {\it null curve} in $\Re$ is a regular, smooth parametrized curve
$\gamma : I\to \Re$, defined on some interval $I\subset \R$, such
that the velocity vector $\gamma'(t)$ is a future-directed null
vector, for each $t\in I$. Since $\gamma'(t)$ is null, $\langle
\gamma'(t), \gamma''(t) \rangle = 0$, and $\gamma''$ has to be
spacelike or proportional to $\gamma'$. Away from flex
points,\footnote{$\gamma(t)$ is a flex point if $\gamma'(t)\wedge
\gamma''(t)= 0$.} the differential 1-form
$$
  \omega_\gamma=\|\gamma''(t)\|^{1/2}dt
   $$
is never zero, and is invariant under changes of parameter and the
action of the Poincar\'e group. The integral of $\omega_\gamma$ can
then be used to introduce a {\it natural parameter} (or
\textit{pseudo-arc parameter}) $s$, intrinsically defined by
$\gamma$, such that
$$
  \|\gamma''(s)\|^{1/2}=1.
   $$
The natural parameter $s$ is fixed up to an additive constant. From
now on we will consider null curves without flex points,
parametrized by the natural parameter.

To any null curve $\gamma(s)$, we associate the \textit{Frenet frame}
\begin{equation}
 \mathbf{F} =(\gamma;\mathbf{t},\mathbf{n},\mathbf{b}) :I\to \P,
   \end{equation}
defined by
$$
  \mathbf{t}(s)=\gamma'(s),\quad \mathbf{n}(s)=
   \gamma''(s),\quad \mathbf{b}(s)
    =\gamma'''(s) + \frac{1}{2}\|\gamma'''(s)\|^2\gamma'(s),
    $$
for every $s\in I$.
The orthonormality conditions \eqref{o.n.conditions} are readily
verified by observing that the derivative of $\langle \gamma', \gamma''\rangle =0$
yields $\langle \gamma', \gamma'''\rangle = -1$.
The Frenet frame satisfies the {\it Frenet--Serret system}
 \begin{equation}\label{FS}
    \gamma'=\mathbf{t},\quad \mathbf{t}'=\mathbf{n},
    \quad \mathbf{n}'=2\k \mathbf{t}+\mathbf{b},\quad \mathbf{b}'=-2\k \mathbf{n},
      \end{equation}
 where the function
$$
  \k(s) := \frac{1}{4}\|\gamma'''(s)\|^2, \quad   s\in I,
    $$
is called the \textit{curvature} of $\gamma$.
Equivalently, \eqref{FS} can be written in the form
\begin{equation}
 \mathbf F' = \mathbf F \left(
          \begin{array}{cccc}
            0 & 0 & 0 & 0 \\
            1 & 0 & -2\k & 0 \\
            0 & 1 & 0 & -2\k \\
            0 & 0 & 1 & 0 \\
          \end{array}
        \right).\label{MCF}
    \end{equation}

\begin{remark}
Let $M = J_3^0(\R,\Re)$ be the space of third order jets of null curves
$\gamma : \R\to \Re$ parametrized by the natural parameter.
The space $M$ is the $\P$-invariant submanifold of
$J_3(\R,\Re)$ defined by
$$
 \dot{\mathbf{x}}\in \mathcal{L}^+,\quad  \|\ddot{\mathbf{x}}\|=1,
  \quad \langle \dot{\mathbf{x}},\ddot{\mathbf{x}}\rangle
    =\langle \ddot{\mathbf{x}},\dddot{\mathbf{x}}\rangle=0,
    \quad \dot{\mathbf{x}}\wedge \ddot{\mathbf{x}}\wedge \dddot{\mathbf{x}}\neq 0,
      $$
where $\mathcal{L}^+\subset \Re$ is the future-directed lightcone. Then
$$
 \rho : M \ni (P,\dot{\mathbf{x}},\ddot{\mathbf{x}},
  \dddot{\mathbf{x}}) \mapsto \left(P;\dot{\mathbf{x}},
   \ddot{\mathbf{x}},\dddot{\mathbf{x}}
    +\frac{1}{2}\|\dddot{\mathbf{x}}\|^2\dot{\mathbf{x}}\right)\in \P
       $$
is a moving frame map in the sense of Fels--Olver \cite{FO}.
Note that the Frenet frame along a null curve $\gamma$ is given
by $\mathbf{F}=\rho\circ j_3(\gamma)$.
\end{remark}

\subsection{Tangent vectors}

Let $\mathcal M$ denote the space of null curves in $\Re$ without
flex points, parametrized by the natural parameter $s$, and
complete, i.e. defined on all $\R$.

\begin{lemma}\label{l:tangent-vector}
Let
\[
  V_\gamma = p_1 \mathbf{t}
   + p_2 \mathbf{n} + p _3\mathbf{b}
    \]
be a vector field along $\gamma \in \mathcal M$. Then $V_\gamma$ is tangent
to $\mathcal M$ if and only if
\begin{equation}\label{tangent-vector}
\left\{
 \begin{array}{l}
  \displaystyle p_2  =  - p_3',\\
   \displaystyle p_1  = \frac{1}{2}p_3''+ \int_0^s
    \kappa' (u) p_3 (u) du + \mathrm{cost}.
     \end{array}
       \right.
         \end{equation}
\end{lemma}

\begin{proof}
Let $\Gamma(s,t)$ be a variation of $\gamma(s)$ through null curves
and assume that $\Gamma(s,0) = \gamma(s)$. Let
$$
 \FR(\cdot,t) =
  (\Gamma;\mathbf{t}_\Gamma,\mathbf{n}_\Gamma,\mathbf{b}_\Gamma)(\cdot,t) : \R \to\P
  $$
denote the Frenet frame for each null curve $\Gamma(\cdot,t)$.
Let
$\k_\Gamma(\cdot,t)$
be the curvature function of $\Gamma(\cdot, t)$.
If we set $\Theta =\FR^{-1} d\FR$, then
\begin{equation}\label{CE0}
 \Theta = K(s,t)ds+ P(s,t)dt
  \end{equation}
for $\mathfrak e(2,1)$-valued functions
$$
 K =\left(
          \begin{array}{cccc}
            0 & 0 & 0 & 0 \\
            1 & 0 & -2\kappa_\Gamma  & 0 \\
            0 & 1 & 0 & -2\kappa_\Gamma  \\
            0 & 0 & 1 & 0 \\
          \end{array}
        \right), \quad
%   $$
%$$
 P =\begin{pmatrix}
  0 & 0 & 0 & 0 \\
  p_1 & p_5 & p_6 & 0\\
    p_2 & p_4 & 0 & p_6\\
     p_3 & 0 & p_4 &-p_5
      \end{pmatrix},
$$
such that
\begin{equation}\label{MCE1}
 d\Theta +\Theta \wedge
  \Theta  = 0
   \end{equation}
is satisfied.
The equation \eqref{MCE1} can be rewritten in the form
\[
 \frac{\partial K}{\partial t}
  - \frac{\partial P}{\partial s} = \left[K ,P  \right],
  \]
which computed at $t=0$ yields
\begin{eqnarray}
  p_2 & = & - p_3', \label{p1}\\
   p_1 & = &\frac{1}{2}p_3''+ \int_0^s
    \kappa_\gamma' (u) p_3 (u) du + \text{cost}, \label{p2}\\
      p_4 &=&- p''_3 -2\kappa_\gamma p_3 + p_1, \label{p3}\\
       p_5 &= &p'_1 +2\kappa_\gamma p'_3,\label{p4}\\
        p_6 & = & p'_5 -2\kappa_\gamma p_4, \label{p5}
         \end{eqnarray}
and
\begin{equation}\label{kappa-t}
 \frac{\partial\kappa_\Gamma}{\partial t} (s,0) =  -\frac{1}{2}p'_6 -\kappa_\gamma p_5,
  \end{equation}
where $\k_\gamma = \k_\Gamma(\cdot, 0)$.
In particular, we have that the infinitesimal variation of $\Gamma(s,t)$ at $t=0$ is
\[
 \frac{\partial\Gamma}{\partial t}(s,0) = p_1 (s) \mathbf{t}_\gamma
    + p_2(s)\mathbf{n}_\gamma +p_3(s)\mathbf{b}_\gamma.
    \]

Conversely, any tangent vector arises as an infinitesimal variation.
In fact, let $V_\gamma$ be a tangent vector at $\gamma$. Let $p_4$, $p_5$
and $p_6$ be the functions determined by $p_1$, $p_2$, $p_3$ through the equations
\eqref{p1}--\eqref{p5}, and consider the function
\[
  c = -\frac{1}{2} p_6' -\kappa_\gamma p_5.
   \]
Under these hypotheses, we have that
\[
 \frac{\partial P}{\partial s} = C - \left[ K_\gamma, P\right],
   \]
where
\[
 C=\left(
          \begin{array}{cccc}
            0 & 0 & 0 & 0 \\
            0 & 0 & -2c & 0 \\
            0 & 0 & 0 & -2c \\
            0 & 0 & 0 & 0 \\
          \end{array}
        \right), \quad
K_\gamma=\left(
          \begin{array}{cccc}
            0 & 0 & 0 & 0 \\
            1 & 0 & -2\kappa_\gamma & 0 \\
            0 & 1 & 0 & -2\kappa_\gamma \\
            0 & 0 & 1 & 0 \\
          \end{array}
        \right).
\]
Next, define
\[
  K(s,t) : = K_\gamma(s)  + t C(s)
   \]
and consider the differential equation for $P = P(s,t)$ given by
\[
  \frac{\partial P}{\partial s} = C - \left[ K_\gamma, P\right] - t \left[C , P \right].
   \]
Let $P(s,t)$ be the solution with initial condition
$P(0,t) = p_0(t)$ and consider the $\mathfrak e(2,1)$-valued 1-form
\[
  K(s,t)ds + P(s,t) dt.
  \]
By construction, this form satisfies the Maurer--Cartan equation, which implies the
existence of an $\P$-valued map
$$
 \FR =
(\Gamma  ; \mathbf{t}, \mathbf{n}, \mathbf{b})
      $$
such that $\FR^{-1} \FR = K(s,t)ds + P(s,t) dt$. As a consequence,
$\Gamma$ represents a variation through null curves whose
infinitesimal variation is $(\partial\Gamma/\partial t) (s,0) =
V_\gamma(p)$.
\end{proof}

\begin{remark}
A solution of \eqref{tangent-vector} is uniquely given by
prescribing arbitrarily the function $p_3$ and a constant. This
gives a canonical trivialization of the tangent bundle $T\mathcal
M$,
$$
   T\mathcal M \simeq \mathcal M \times \R \times C^\infty(\R,\R).
    $$
\end{remark}

\section{Local motion of null curves and evolution equations}

\subsection{Local motion and vector fields}

An \textit{invariant local motion of null curves} is an integral curve
of a \textit{local vector field} on $\mathcal M$, that is a section
of $T\mathcal M$ of the form
\begin{equation}\label{local-vf}
 V : \mathcal M \ni \gamma \mapsto
  V_\gamma = \p_1[\kappa_\gamma] \mathbf t_\gamma + \p_2[\kappa_\gamma] \mathbf n_\gamma
  + \p_3[\kappa_\gamma] \mathbf b_\gamma,
    \end{equation}
where $\p_1$, $\p_2$, $\p_3$ are polynomial differential functions satisfying
\begin{equation}\label{Ep3}
 E(u_{(1)}\p_3)=0
  \end{equation}
and
\begin{equation}\label{p1-p2}
  \p_1=\frac{1}{2}D^2 \p_3+\int u_{(1)}\p_3 + \text{cost}, \quad \p_2=-D \p_3.
  \end{equation}

\begin{remark}
Note that, according to Lemma \ref{l:tangent-vector}, \eqref{Ep3} and
\eqref{p1-p2} are necessary and sufficient conditions for
$V_\gamma$ being tangent to  $\mathcal M$ at $\gamma$, for each
$\gamma \in \mathcal M$. Moreover, a local vector field is completely
determined by a differential function $\p_3$ such that $E(u_{(1)}\p_3)=0$
and a constant. Henceforth, such a constant will be assumed to be zero.
From \eqref{local-vf}, it follows that a local motion of null curves
is a solution $\Gamma(s,t)$ of the flow equation
\begin{equation}\label{flow-equation}
 \frac{\partial \Gamma}{\partial t} = \p_1[\kappa_\Gamma] \mathbf t_\Gamma
+ \p_2[\kappa_\Gamma] \mathbf n_\Gamma
  + \p_3[\kappa_\Gamma] \mathbf b_\Gamma.
   \end{equation}
\end{remark}

We can now state the following.

\begin{thm}\label{thm:local-vf}
Suppose $\Gamma(s,t)$ is the local motion associated with a local vector field
$V$. Then the evolution of the curvature is governed by
\begin{equation}\label{kappa-deltastorto-erre}
 \frac{\partial \kappa}{\partial t}
  = -\frac{1}{4}\mathcal D D^{-1} \mathcal D \p_3[\kappa],
  \end{equation}
where $\p_3$ is a polynomial differential function such that $u_{(1)}\p_3$ is
a total derivative.

Conversely, let $\kappa(s,t)$ be a solution of \eqref{kappa-deltastorto-erre}
and let $\gamma$ be the null curve
with curvature $\kappa(s,0)$. Then there exists a unique local motion $\Gamma(s,t)$
such that $\Gamma(s,0) = \gamma$ which corresponds to the
local vector field determined by $\p_3$.
\end{thm}

\begin{proof}
As above, the flow $\Gamma$ of the local vector field $V$
lifts to a map $\mathbf F : \R^2 \to \P$ and there exist
$\mathfrak e(2,1)$-valued polynomial differential functions
$$
 \mathfrak K=\left(
          \begin{array}{cccc}
            0 & 0 & 0 & 0 \\
            1 & 0 & -2u_{(0)} & 0 \\
            0 & 1 & 0 & -2u_{(0)} \\
            0 & 0 & 1 & 0 \\
          \end{array}
        \right),
\quad
\mathfrak P =\begin{pmatrix}
  0 & 0 & 0 & 0 \\
  \p_1 & \p_5 & \p_6 & 0\\
    \p_2 & \p_4 & 0 & \p_6\\
     \p_3 & 0 & \p_4 & -\p_5
      \end{pmatrix}
        $$
so that the $\mathfrak e(2,1)$-valued 1-form
\begin{equation}\label{m-c-local}
 \Theta =\mathbf F^{-1}d \mathbf F = \mathfrak K[\kappa] ds
  + \mathfrak P[\kappa]dt,
   \end{equation}
satisfies the Maurer--Cartan equation
$d\Theta  + \Theta  \wedge \Theta  =0$.
 Writing out this equation yields
\begin{equation}\label{LM1}
 \left\{ \begin{array}{lll}
  \p_1 = \frac{1}{2}D^2 \p_3+\int u_{(1)}\p_3,\\
     \p_2 = -D \p_3,\\
      \p_4 = -\frac{1}{2}D^2 \p_3 -2u_{(0)}\p_3+ \int u_{(1)}\p_3 ,\\
       \p_5=D \p_1+2u_{(0)}D \p_3,\\
        \p_6 = D \p_5 - 2u_{(0)}\p_4,
         \end{array}\right.
          \end{equation}
and
\begin{equation}\label{LM2}
 \frac{\partial \kappa}{\partial t} = \frac{1}{2} \mathcal D \p_4.
    \end{equation}
From the third equation of \eqref{LM1},
and the hypothesis that $E(u_{(1)}\p_3)=0$,
it follows that
\begin{equation}\label{d-p4=dstorto-p3}
  D \p_4 = - \frac{1}{2}\mathcal D \p_3,
  \end{equation}
and hence \eqref{LM2} can be written in the form
\[
 \frac{\partial \kappa}{\partial t} = -\frac{1}{4} \mathcal D D^{-1} \mathcal D \p_3.
   \]

Conversely, if $\k : \R^2 \to \R$ is a solution to equation
\eqref{kappa-deltastorto-erre}, then the
$\mathfrak e(2,1)$-valued 1-form $\Theta$ defined as in \eqref{m-c-local} satisfies the
Maurer--Cartan equation.
Thus, by the Cartan--Darboux theorem, there exists a map
$$
  \mathbf{F} = (\Gamma;\mathbf{a}) : \R^2 \to \P,
    $$
unique up to left multiplication by an element of $\P$, such that
$\mathbf{F}^{-1}d\mathbf{F}=\Theta$. Consequently,
$\Gamma : \R^2 \to \Re$ defines the flow of a local vector field.
\end{proof}

\begin{remark}
The curvature evolution of a local vector field is given by
a local evolution equation, that is
\[
 \frac{\partial \kappa}{\partial t} = \mathfrak c[\kappa],
  \]
where $\mathfrak c \in J[u]$ is a polynomial differential function.
Remarkably, the right hand side of \eqref{kappa-deltastorto-erre}
is expressed in terms of the two differential operators
defining the bi-Hamiltonian structure of the KdV equation (cf. \cite{O2}).
\end{remark}

\section{Dynamics of null curves and the KdV hierarchy}

\subsection{Motion by integrable evolution equations}

We now show that the evolution of the curvature induced by certain
motions is completely integrable. This is the content of the
following.

\begin{thm}\label{thm:kdv hier}
The sequence of polynomial differential functions
\begin{equation}\label{GP-sequence}
 \p_3^{[n]} = 4 \mathfrak g_{n-2},\quad n \geq 2,
   \end{equation}
defines a hierarchy of local motions whose curvature evolution equations
are the members of the KdV hierarchy.
\end{thm}

\begin{proof}
Since any $\p_3^{[n]}$ in the sequence is the variational derivative of some
local functional, $\kappa_{(1)} \p_3^{[n]}$ is a
total derivative (cf. \cite{O3}, Theorem 7.36), or equivalently
$E(\kappa_{(1)} \p_3)$ $= 0$. Therefore, $\p_3^{[n]}$ determines a local vector field.
The result now follows from the curvature equation \eqref{kappa-deltastorto-erre}
and the Lenard recursion formula \eqref{lenard-formula} for the $\mathfrak g_n$'s.
For $n=2$,
\[
 \kappa_t = -\frac{1}{4}\mathcal D D^{-1} \mathcal D 4\,\mathfrak g_0 =
  -\mathcal D \mathfrak g_1 = -D \mathfrak g_2 = -6uu_s - u_{sss},
  \]
which is the KdV equation, the second member of the KdV hierarchy. For $n>2$,
\[
 \kappa_t = -\frac{1}{4}\mathcal D D^{-1} \mathcal D \,4 \mathfrak g_{n-2} =
  -\mathcal D \mathfrak g_{n-1} = -D \mathfrak g_n,
   \]
which is the $n^{th}$ member of the KdV hierarchy,

\end{proof}

\begin{remark}
Let $\mathcal M \simeq \E(2,1) \times C^\infty(\R,\R)$ be the space of null curves
parametrized by the natural parameter and defined on the whole $\R$.
Assume, for simplicity, that $C^\infty(\R,\R)$ is the space of
rapidly decreasing smooth functions, or of smooth periodic functions
of period $2\pi$.
Consider the space $\widehat{\mathcal M} \simeq \mathcal M/ \P \simeq C^\infty$
of ``geometric'' null curves (the space of curvature functions).
If we think of a tangent vector
$v_c \in T_\gamma \mathcal M$ as defined by a smooth function $c$ and
an element of the
Lie algebra $\mathfrak e(2,1)$, so that
$T_\gamma \mathcal M\simeq \mathfrak e(2,1) \times C^\infty$,
it is clear that two tangent vectors $v_c, v_{\tilde c} \in T_\gamma \mathcal M$
descend to the same element
of $T_{[\gamma]}\widehat{\mathcal M}$
if and only if $c = \tilde c$. This allows us to define a symplectic 2-form $\Omega$
on $\widehat{\mathcal M}$ by
\[
  \Omega(c, \tilde c) = \frac{1}{2}\int \left( \tilde c \, D^{-1} c
  -c \,D^{-1} \tilde c\right) ds.
     \]
If $\mathcal F = \int F(\kappa, \kappa_s, \kappa_{ss}, \dots) ds$ is a
\textit{local} functional on $C^\infty$,
then
\[
  \frac{d}{d\epsilon}_{\vert_{\epsilon =0}} \mathcal F(c+\epsilon \tilde c) =
   \left(\frac{\delta \mathcal F}{\delta \kappa}(c), \tilde c \right) =
     \Omega \left(D \frac{\delta \mathcal F}{\delta \kappa}(c), \tilde c
       \right),
         \]
where $(\,,\,)$ denotes the $L^2$ inner product
and ${\delta \mathcal F}/{\delta \kappa}$ is the variational derivative of
$\mathcal F$.
The symplectic gradient (Hamiltonian vector field)
of $\mathcal F$ is given by
$D \frac{\delta \mathcal F}{\delta \kappa}$, which defines
the Hamiltonian flow $\frac{\partial \kappa}{\partial t} =
D \frac{\delta \mathcal F}{\delta \kappa}$.
The corresponding Poisson bracket is given by
\[
   \{\mathcal F_1, \mathcal F_2\} =
     \Omega \left(D \frac{\delta \mathcal F_1}{\delta \kappa},
      D \frac{\delta \mathcal F_2}{\delta \kappa}
        \right) =
  \left(\frac{\delta \mathcal F_1}{\delta \kappa},
  D \frac{\delta \mathcal F_2}{\delta \kappa}   \right)
    =
   \int \frac{\delta \mathcal F_1}{\delta \kappa}
   D \frac{\delta \mathcal F_2}{\delta \kappa} ds.
      \]
%A similar construction can be given also for the operator $\mathcal D$.

A local vector field on $\mathcal M$ is said \textit{Hamiltonian}
with respect to a differential operator $\mathcal J$
if its curvature equation can be written in the form
\begin{equation}\label{motion-hamiltonian}
 \frac{\partial \kappa}{\partial t} = \mathcal
   J \frac{\delta \mathcal F}{\delta \kappa}[\kappa],
     \end{equation}
where ${\delta \mathcal F}/{\delta \kappa}$ is the variational derivative of
$\mathcal F = \int F(\kappa, \kappa_s, \kappa_{ss}, \dots) ds$
and $\mathcal J$ determines the Poisson bracket
$\{\mathcal F_1, \mathcal F_2\} = \int \frac{\delta \mathcal F_1}{\delta \kappa}
    \mathcal J \frac{\delta \mathcal F_2}{\delta \kappa} ds
      $
on the space of functionals.

Note that the local vector fields of Theorem
\ref{thm:kdv hier} are Hamiltonian with respect to the two
Hamiltonian structures $D$ and $\mathcal D$, which coincide with the
canonical Hamiltonian structures of the KdV.
\end{remark}

\subsection{Evolution by congruences and null worldlines}

A null curve that moves without changing its
shape (by Lorentz rigid motion) under the KdV flow
\begin{equation}\label{KdV}
 \k_t+\k_{sss}+6\k\k_s=0
  \end{equation}
is said a {\it congruence curve}. Congruence curves correspond to
the traveling wave solutions of the KdV equation. If $\k(s,t)= f
(s-\lambda t)$ is a traveling wave solution of (\ref{KdV}), then
\begin{equation}\label{KDV2}
 f'''+ 6ff' - \lambda f'=0.
  \end{equation}
Integrating twice,
we obtain
$$
  (h')^2=4h^3-g_2h-g_3,
   $$
for real constants $g_2$ and $g_3$, where
$$
  h = -\frac{1}{2}\left(f-\frac{\lambda}{6}\right).
  $$
Thus $f$ can be expressed in
terms of the Weierstrass $\wp$-function with invariants $g_2$ and
$g_3$. Quasi-periodic congruence curves may occur only if
the polynomial $p(x) = 4x^3-g_2x-g_3$ has three distinct real roots, i.e.
$$
  27g_3^2-g_2^3<0.
    $$
In this case, the periodic solution of (\ref{KDV2}) is
$$
  f(s)=-2\wp(s+\omega_3;g_2,g_3)+\frac{\lambda}{6},\quad s\in \R,
     $$
where $\omega_1$, $\omega_3$ are the primitive half-periods of
$\wp(\,\cdot\,;g_2,g_3)$ (cf. \cite{L}).

Now, the third order ODE (\ref{KDV2}) coincides with
the Euler-Lagrange equation
of the action functional on null curves defined by
\begin{equation}\label{KDV4}
 \int(2\k + \lambda)ds
   \end{equation}
(cf. \cite{FGL}, \cite{GM}). Thus, congruence curves are the
worldlines of the relativistic particle model defined by
(\ref{KDV4}). (For more details on the physical models associated
with action functionals of the type above see \cite{KP}, \cite{MN1},
\cite{MN2}, \cite{NMMK}, \cite{P}.)

More interestingly, let
$\k$ be a solution of (\ref{KDV2}) and let $\mathfrak K$ and $\mathfrak P$ be
as in \eqref{m-c-local}, where $\mathfrak P$ is computed for
$\p_3^{[2]} = 4\mathfrak g_0 = 2$, using \eqref{LM1}.
If we set
$$
 \mathfrak{L}_{\lambda} := \mathfrak{P}[\kappa]+\lambda \mathfrak{K}[\kappa],
  $$
then from the Maurer--Cartan equation of $\Theta =\mathfrak K ds +\mathfrak Pdt$,
which in turn is equivalent to
\[
 \frac{\partial \mathfrak P[\kappa]}{\partial s}
 -\frac{\partial \mathfrak K[\kappa]}{\partial t}
   = [\mathfrak P, \mathfrak K][\kappa],
   \]
it follows that
\begin{equation}\label{LP}
 \frac{\partial \mathfrak L_\lambda}{\partial s}
  = \left[\mathfrak{L}_{\lambda},\mathfrak{K}\right].
   \end{equation}
Thus, $\mathfrak{L}_{\lambda}$ and $\mathfrak{K}$ form a Lax pair
for the variational problem \eqref{KDV4}. Equation (\ref{LP}) means
that the linear endomorphism
$$
 \mu = \FR\cdot
   \mathfrak{L}_{\lambda}\cdot \FR^{-1}
    $$
is constant along the solutions. From this it
follows that the worldlines can be obtained by quadratures and
expressed in terms of Weierstrass $\sigma$, $\zeta$ and $\wp$
functions (cf. \cite{GM}).

Summarizing, we can state the following.

\begin{thm}\label{thm:cong}
The congruence solutions of the flow generated by $\p_3^{[2]}$
coincide with the worldlines of the particle model defined by
\eqref{KDV4}. In particular, the congruence curves are integrable
by quadrature.
\end{thm}

\begin{remark}
More generally, a local motion is \textit{by congruences}
if the curves of the motion do not change their shapes during the evolution.
This means that the curvature of the motion is a traveling wave solution
of equation \eqref{kappa-deltastorto-erre}, i.e., $\kappa(s,t) = f(s-\lambda t)$, for
some constant $\lambda$, and $f$ is a solution to
\begin{equation}\label{stationary-eq}
4\lambda f'
  - \mathcal D D^{-1} \mathcal D \p_3[f] = 0.
 \end{equation}
In this case, if we set
$$
 \mathfrak{L}_{\lambda} := \mathfrak{P}[f]+\lambda \mathfrak{K}[f],
  $$
then the Maurer--Cartan equation of $\Theta =\mathfrak K ds +\mathfrak Pdt$
can be written in Lax form
\[
 (\mathfrak L_\lambda)'
  = \left[\mathfrak{L}_{\lambda},\mathfrak{K}\right].
   \]
Unlike the local motion discussed in Theorem \ref{thm:cong}, in general
the Lax formulation does not imply the integration by quadratures of
congruence curves.

\end{remark}

\subsection{Evolution by similarities} Consider a null curve $\gamma:I\to \Re$
parametrized by the natural parameter, and let $\k_\gamma$ be its curvature. Then
$$
 \widetilde{\gamma}:s\in \sqrt{r}I\to r\gamma(\frac{s}{\sqrt{r}})\in \Re
  $$
is the natural parametrization of $r\gamma$ and its curvature is given by
$$
  \widetilde{\k}_{\tilde \gamma}(s)=\frac{1}{r}\k_\gamma(\frac{s}{\sqrt{r}}).
   $$
Thus, solutions of (\ref{KdV}) which corresponds to curves whose shapes
evolves by similarities under the KdV flow are in the form
$$
  \k(s,t)=\frac{1}{r(t)}\k_\gamma \left(\frac{s}{\sqrt{r(t)}} \right)
   $$
where $r$ is a positive smooth function. Setting $x=s/\sqrt{r(t)}$ we have
$$
  2\k_\gamma'''(x)-x\sqrt{r(t)}\dot{r}(t)\k_\gamma'(x)+2\k_\gamma(x)\left(6\k_\gamma'(x)
   -\sqrt{r(t)}\dot{r}(t)\right)=0.
   $$
This implies
$$
  \frac{d}{dt}\left(\sqrt{r}\dot{r}\right)|_t(xk_\gamma'(x)+2\k_\gamma(x))=0.
   $$
Excluding the trivial case $\k_\gamma(s)=-2s^{-2}$ (which corresponds to a stationary solution),
 we obtain
$$
  r(t)=(at+b)^{2/3}.
    $$
The resulting third order ODE for $\k_\gamma$ is
\begin{equation}\label{SS}
  \k_\gamma'''+6\k_\gamma\k_\gamma'-\frac{a}{3}(x\k_\gamma'+2\k_\gamma)=0.
     \end{equation}
Without loss of generality we may assume $a=1$. Thus, (\ref{SS}) can be integrated by
setting
$$
  \k_\gamma=v'-v^2,
    $$
where $v$ is a solution of the {\it second Painlev\'e equation}
$$
   v''-2v^3+xv+c=0.
     $$
This explains the geometrical origin of the similarity reduction of the KdV
equation considered in \cite{JO}. We have proved the following.

\begin{thm}\label{thm:sim}
The curvature of the similarity solutions corresponding to the
flow generated by $\p_3^{[2]}$
can be integrated by means of the solutions of the second
Painlev\'e equation.
\end{thm}

\bibliographystyle{amsalpha}

\end{document}